\documentclass[11pt, article]{amsart}
\usepackage{url,xcolor}
\usepackage{amscd}
\usepackage[centertags]{amsmath}
\usepackage{latexsym}
\usepackage{amsfonts}
\usepackage{amssymb}
\usepackage{amsthm}
\usepackage{newlfont}
\usepackage{graphics}
\usepackage{graphicx}
\usepackage[all]{xy}
\usepackage{epstopdf}
\usepackage[dvipdfmx]{attachfile2}
\def\R{\mathbb{R}}

\theoremstyle{plain}
\newtheorem{thm}{Theorem}

\newtheorem{lem}[thm]{Lemma}

\theoremstyle{remark}

\theoremstyle{definition}

\begin{document}

\date{June 2017}

\title{Multiples of long period small element continued fractions to short period large elements continued fractions}

\author{Michael O. Oyengo}
\address{Department of Mathematics \\
University of Illinois\\
1409 W. Green Street \\
Urbana, IL 61801}
\email{mchlyng2@illinois.edu}


\begin{abstract}
We construct a class of quadratic irrationals having continued fractions of period $n\geq2$ with ``small'' partial quotients for which certain integer multiples have continued fractions of period $1$, $2$ or $4$ with ``large'' partial quotients. We then show that numbers in the period of the new continued fraction are simple functions of the numbers in the periods of the original continued fraction. We give generalizations of some of the continued fractions and show that polynomials arising from the generalizations are related to Chebyshev and Fibonacci polynomials. We then show that some of these polynomials have a hyperbolic root distribution.
%We study continued fractions of quadratic irrationals related to Fibonacci type numbers. In particular, we show that integer multiples these continued fractions can be described by a simple continued fraction. We also show how polynomials arising from generalizations of these continued fractions are related to Chebyshev and Fibonacci polynomials, and in some cases have hyperbolic root distribution.
\end{abstract}

\maketitle

%%%%%%%%%%%%%%%%%%%%%%%%%%%%%%%%%%%%%%%%%%%%%%%%%%%%%%%%%%%%%%%%%%%%%%%%%
% Macros
%%%%%%%%%%%%%%%%%%%%%%%%%%%%%%%%%%%%%%%%%%%%%%%%%%%%%%%%%%%%%%%%%%%%%%%%%

\newcommand\sfrac[2]{{#1/#2}}

\newcommand\cont{\operatorname{cont}}
\newcommand\diff{\operatorname{diff}}

%%%%%%%%%%%%%%%%%%%%%%%%%%%%%%%%%%%%%%%%%%%%%%%%%%%%%%%%%%%%%%%%%%%%%%%%%

\section{Introduction}
Let $[a_{0}, a_{1}, a_{2},\dots]$ be a simple continued fraction expansion of a real number $\alpha$ where $a_{j}$'s are positive integers. If $\alpha$ is a quadratic irrational, then by a theorem of Lagrange (see for example \cite[p. 44]{Bor}) its continued fraction will be periodic, i.e.,
$$\alpha=[a_{0}, a_{1},\dots,a_{k},\overline{b_{1},b_{2},\dots,b_{n}}],$$
where $b_{1},b_{2},\dots,b_{n}$ is the period of the expansion of $\alpha$, and $a_{0}, a_{1},\dots,a_{k}$ is the non-periodic part. We say that $\alpha$ is of period $k$ if its periodic part has length $k$.

If $\alpha$ is a quadratic irrational, then clearly for any positive integer $N$, $N\alpha$ is still a quadratic irrational, and its continued fraction will be periodic. In \cite{Cus}, Cusick  presents an algorithm for obtaining the continued fraction of $N\alpha$ and uses it to give estimates for the length of the period of the expansion of $N\alpha$ in relation to that of $\alpha$. In this paper, we are going to construct a quadratic irrational $\alpha$ such that when $N$ is a Fibonacci or Lucas number, the continued fraction of $N\alpha$ has period of length $1$, $2$ or $4$. As we will see, the length of the period of the new continued fraction depends on the parity of $n$, the length of the period of the original continued fraction.

Fibonacci numbers and Lucas numbers are integers that solve the the recurrence relation
$$f_{n+1}=f_{n}+f_{n-1}$$
under the initial conditions $F_{0}=0$, $F_{1}=F_{2}=1$, $L_{0}=2$ and $L_{1}=1$. Fibonacci numbers can also be defined as a particular case $F_{n}(1)$, of Fibonacci polynomials which are generated by the rational function
\begin{equation}\label{fibgen}
  \frac{t}{1-xt-t^{2}}=\sum_{k=1}^{\infty}F_{k}(x)t^{k}.
\end{equation}
In section \ref{sec2}, we construct Fibonacci-like numbers $\tilde{F}_{n}(m)$ and use them to construct $\alpha_{n}$, the largest root of the quadratic $$x^{2}-2NF_{n}x-F_{n} \tilde{F}_{n}(2N)$$ where $N$ is a positive integer. We then show the relationship between
the continued fraction of $\alpha_{n}$ and $$F_{n}\cdot[\overline{2N,1^{(n-1)}}]$$ where $1^{(n)}=1,1,1,\dots,1$ $n$ times.

In section \ref{sec3}, we construct Lucas-like numbers $\tilde{L}_{n}(m)$ and use them to construct $\beta_{n}$, the largest root of the quadratic $$y^{2}-2NL_{n}y-L_{n} \tilde{L}_{n}(2N)$$ where $N=5k+3$ is a positive integer. We then show the relationship between
the continued fraction of $\beta_{n}$ and $$L_{n}\cdot[\overline{2N,1^{(n-2)},2,1,2k,1,2,1^{(n-2)}}].$$

In section \ref{sec4}, we generalize the continued fractions, as well as results of section \ref{sec2}. Polynomials arising from convergents of the generalized continued fractions are studied in section  \ref{sec5}. In particular, we show how these polynomials relate to Fibonacci and Chebyshev polynomials, and show that some of the polynomials have their roots on hyperbolas.

%%%%%%%%%%%%%%%%%%%%%%%%%%%%%%%%%%%%%%%%%%%%%%%%%%%%%%%%%%%%%%%%%%%%%%%%%%%%%%%%%%%%%%%%%%%%%%%%%%%%%%%%%%%%%%%%%%%%%%%%%%%%%%%%%%%%%%%%%%%%%%%%%%%%%%%%%%%%

\section{Fibonacci-like numbers and quadratic irrationals}\label{sec2}

Let $m$ be a positive integer and define the sequence $\tilde{F}_{n}(m)$ by the recurrence relation
\begin{equation}\label{fib1}
  \tilde{F}_{n}(m)= \tilde{F}_{n-1}(m)+ \tilde{F}_{n-2}(m)
\end{equation}
under the initial conditions $\tilde{F}_{0}(m)=0$,  $\tilde{F}_{1}(m)=1$ and $\tilde{F}_{2}(m)=m$. It is easy to see that they are generated by the rational function
\begin{equation}\label{fib2}
  \frac{t+(m-1)t^{2}}{1-t-t^{2}}=\sum_{n=0}^{\infty}\tilde{F}_{n}(m)t^{n}.
\end{equation}
Clearly, when $m=1$ we get the usual Fibonacci sequence. From the generating function (\ref{fib2}), and the generating function of Fibonacci numbers, we get the relation
\begin{eqnarray}
 \nonumber \tilde{F}_{n+1}(m) &=& F_{n+1}+(m-1)F_{n} \\ \label{fib2a}
  &=& mF_{n}+F_{n-1}.
\end{eqnarray}
For a fixed positive integer $N$, consider the quadratic
\begin{equation}\label{fib3}
  x^{2}-2NF_{n}x-F_{n}\tilde{F}_{n}(2N).
\end{equation}
The roots of the quadratic are given by
\begin{eqnarray}\label{fib4}
  \nonumber\alpha_{n}(N) &=& NF_{n}+\sqrt{N^{2}F_{n}^{2}+F_{n}\tilde{F}_{n}(2N)},\\
  \overline{\alpha}_{n}(N) &=& NF_{n}-\sqrt{N^{2}F_{n}^{2}+F_{n}\tilde{F}_{n}(2N)}.
\end{eqnarray}
We study the continued fraction of these quadratic irrationals.

A quadratic irrational $\alpha$ is said to be reduced if $\alpha>1$ and $-1/\overline{\alpha}>1$. If $\alpha$ is reduced, then its continued fraction is purely periodic (see \cite{Bor}  Theorem 2.48). First we show that for $n$ odd, $1/(\alpha_{n}(N)-\tilde{F}_{n+1}(2N))$ is reduced.

Let $B_{n}(N)=N^{2}F_{n}^{2}+F_{n}\tilde{F}_{n}(2N)$,
\begin{eqnarray*}
% \nonumber to remove numbering (before each equation)
  1/(\alpha_{n}(N)-\tilde{F}_{n+1}(2N)) &=& -1/(NF_{n}+F_{n-1}-\sqrt{B_{n}(N)}) \\
  &=&  \frac{-(NF_{n}+F_{n-1}+\sqrt{B_{n}(N)})}{F_{n-1}^{2}-F_{n}F_{n-2}} \\
   &=& NF_{n}+F_{n-1}+\sqrt{B_{n}(N)} >1.
\end{eqnarray*}
On the other hand,
$$\tilde{F}_{n+1}(2N)-\overline{\alpha}_{n}(N) = NF_{n}+F_{n-1}+\sqrt{B_{n}(N)} >1.$$
We have used (\ref{fib2a}) and the well known identity for Fibonacci numbers
\begin{equation}\label{fid}
  F_{n-1}^{2}-F_{n}F_{n-2}=(-1)^{n}.
\end{equation}
% $$ which is easy to prove by induction on $n$ and by the recurrence relation for Fibonacci numbers.
Let $x_{n}(N) = NF_{n}+F_{n-1}+\sqrt{B_{n}(N)}$. By using (\ref{fib2a}) and the Euclidean algorithm we get
\begin{eqnarray*}
% \nonumber to remove numbering (before each equation)
  x_{n}(N) &=& 2\tilde{F}_{n+1}(N)-(NF_{n}+F_{n-1}-\sqrt{B_{n}(N)}) \\
   &=& 2\tilde{F}_{n+1}(N)+\frac{1}{x_{n}(N)}.
\end{eqnarray*}
We now have
\begin{eqnarray*}
% \nonumber to remove numbering (before each equation)
  \alpha_{n}(N) &=& x_{n}-F_{n-1}\\
   &=&\tilde{F}_{n+1}(2N)+\frac{1}{x_{n}(N)} \\
   &=& [\tilde{F}_{n+1}(2N),\;\overline{2\tilde{F}_{n+1}(N)}].
\end{eqnarray*}

For $n$ even, the calculations for the continued fraction expansion of $\alpha_{n}$ are a bit involving due to the identity (\ref{fid}).
\begin{eqnarray*}
% \nonumber to remove numbering (before each equation)
  1/(\alpha_{n}(N)-\tilde{F}_{n+1}(2N)+1) &=& \frac{1}{1-(NF_{n}+F_{n-1}-\sqrt{B_{n}(N)})} \\
  &=&  \frac{NF_{n}+F_{n-1}+\sqrt{B_{n}(N)}}{NF_{n}+F_{n-1}+\sqrt{B_{n}(N)}-(F_{n-1}^{2}-F_{n}F_{n-2})} \\
   &=& \frac{NF_{n}+F_{n-1}+\sqrt{B_{n}(N)}}{NF_{n}+F_{n-1}+\sqrt{B_{n}(N)}-1} >1.
\end{eqnarray*}
We also have that
$$\tilde{F}_{n+1}(2N)-1-\overline{\alpha_{n}} = NF_{n}+F_{n-1}+\sqrt{B_{n}(N)} -1>1.$$
Now let $x_{n}(N) = 1/(1-(NF_{n}+F_{n-1}-\sqrt{B_{n}(N)}))$,
\begin{eqnarray*}
% \nonumber to remove numbering (before each equation)
 % x_{n} &=& 1+\frac{NF_{n}+F_{n-1}-\sqrt{B_{n}(N)}}{1-(NF_{n}+F_{n-1}-\sqrt{B_{n}(N)})} \\
    %&=& 1+\frac{F_{n-1}^{2}-F_{n}F_{n-2}}{NF_{n}+F_{n-1}+\sqrt{B_{n}(N)}-(F_{n-1}^{2}-F_{n}F_{n-2})}   \\
  x_{n}(N)  &=& 1+\frac{1}{NF_{n}+F_{n-1}+\sqrt{B_{n}(N)}-1}   \\
    &=&  1+\frac{1}{2\tilde{F}_{n+1}(N)-2+(1-(NF_{n}+F_{n-1}-\sqrt{B_{n}(N)}))}  \\
    &=&  1+\frac{1}{2\tilde{F}_{n+1}(N)-2+1/x_{n}(N)}.
\end{eqnarray*}
This implies that for $n$ even,
\begin{eqnarray*}
% \nonumber to remove numbering (before each equation)
  \alpha_{n}(N) &=& \tilde{F}_{n+1}(2N)-1+\frac{1}{x_{n}(N)} \\
   &=& [\tilde{F}_{n+1}(2N)-1,\;\overline{1,2\tilde{F}_{n+1}(N)-2}].
\end{eqnarray*}
In each case, the length of the period in the continued fraction expansion of $\alpha_{n}(N)$ is independent of $n$.  We have proved

\begin{thm}\label{alpha}
Let $\alpha_{n}(N)$ be the largest root of the quadratic (\ref{fib3}) and $N>1$ be a positive integer, then the continued fraction of $\alpha_{n}(N)$ for $n$ odd is given by
\begin{equation}\label{fib5}
  \alpha_{n}(N)=[\tilde{F}_{n+1}(2N),\overline{2\tilde{F}_{n+1}(N)}]
\end{equation}
while for $n$ even it is given by
\begin{equation}\label{fib6}
  \alpha_{n}(N)=[\tilde{F}_{n+1}(2N)-1,\overline{1,2\tilde{F}_{n+1}(N)-2}].
\end{equation}
\end{thm}

%What can be said of the continued fraction of

We now show the relationship between the continued fraction of $\alpha_{n}(N)$ and $F_{n}\cdot[\overline{2N,1^{(n-1)}}]$ where $1^{(n)}=1,1,\dots,1$ $n$ times. Let $\alpha$ have the simple continued fraction
\begin{equation}\label{con1}
  \alpha=[c_{0},c_{1},c_{2},\dots,c_{k-1},c_{k}] = \frac{p_{k}}{q_{k}}
\end{equation}
%It follows that
%$$\alpha = \frac{\alpha p_{k}+p_{k-1}}{\alpha q_{k}+q_{k-1}}$$
where $p_{0}=a_{0}=\lfloor\alpha\rfloor,\;p_{-1}=1,\;q_{0}=1$ and $q_{-1}=0$. We will use the correspondence (see \cite[Lemma 2.8]{Bor})
\begin{equation}\label{con2}
  \left(\begin{array}{cc}c_{0} & 1 \\ 1 & 0 \\ \end{array}\right)\left(\begin{array}{cc}c_{1} & 1 \\ 1 & 0 \\ \end{array}\right)\ldots \left(\begin{array}{cc}c_{k-1} & 1 \\ 1 & 0 \\ \end{array}\right)\left(\begin{array}{cc}c_{k} & 1 \\ 1 & 0 \\ \end{array}\right) = \left(\begin{array}{cc}p_{k} & p_{k-1} \\ q_{k} & q_{k-1} \\ \end{array}\right).
\end{equation}
Taking the determinant on both sides of (\ref{con2}) gives the identity
\begin{equation}\label{con3}
  p_{k}q_{k-1}-p_{k-1}q_{k}=(-1)^{k+1}.
\end{equation}
Also, taking the transpose of the matrices on both sides of  (\ref{con2}) we get
\begin{equation}\label{con4}
  \left(\begin{array}{cc}c_{k} & 1 \\ 1 & 0 \\ \end{array}\right)\left(\begin{array}{cc}c_{k-1} & 1 \\ 1 & 0 \\ \end{array}\right)\ldots \left(\begin{array}{cc}c_{1} & 1 \\ 1 & 0 \\ \end{array}\right)\left(\begin{array}{cc}c_{0} & 1 \\ 1 & 0 \\ \end{array}\right) = \left(\begin{array}{cc}p_{k} & q_{k} \\ p_{k-1} & q_{k-1} \\ \end{array}\right)
\end{equation}
from which we deduce
\begin{eqnarray}
% \nonumber to remove numbering (before each equation)
 \nonumber [c_{k},c_{k-1},\ldots,c_{1},c_{0}] &=& \frac{p_{k}}{p_{k-1}} \\
  \label{con5} [c_{k},c_{k-1},\ldots,c_{1}] &=& \frac{q_{k}}{q_{k-1}}
\end{eqnarray}

\begin{lem}\label{lem2}
For all $n\geq2$, we have the continued fraction,
\begin{equation}\label{ratio1}
  \frac{\tilde{F}_{n+1}(m)}{\tilde{F}_{n}(m)}=[1^{(n-1)},m]
\end{equation}
and for $n\geq3$ we have the continued fraction
\begin{equation}\label{ratio1}
  \frac{\tilde{F}_{n+2}(m)}{\tilde{F}_{n}(m)}=[2,1^{(n-2)},m]
\end{equation}
where $1^{(n)}=1,1,\dots,1$ $n$ times.
\end{lem}
\begin{proof}
We can easily show by induction on $k$ that for all $k\geq 1$,
 $$\left(\begin{array}{cc}1 & 1 \\ 1 & 0 \\ \end{array}\right)^{k}=\left(\begin{array}{cc} F_{k+1} & F_{k} \\ F_{k} & F_{k-1} \\ \end{array}\right).$$
It follows that
$$\left(\begin{array}{cc}1 & 1 \\ 1 & 0 \\ \end{array}\right)^{k-1}\left(\begin{array}{cc}m & 1 \\ 1 & 0 \\ \end{array}\right)=\left(\begin{array}{cc} \tilde{F}_{k+1}(m) & _{k} \\ \tilde{F}_{k}(m) & F_{k-1} \\ \end{array}\right)$$
where we have used (\ref{fib2a}). It follows by (\ref{con1}), and the correspondence (\ref{con2}) that $[1^{(k-1)},m]=\frac{\tilde{F}_{k+1}(m)}{\tilde{F}_{k}(m)}.$

Similarly,
\begin{eqnarray*}
% \nonumber to remove numbering (before each equation)
 \left(\begin{array}{cc}2 & 1 \\ 1 & 0 \\ \end{array}\right) \left(\begin{array}{cc}1 & 1 \\ 1 & 0 \\ \end{array}\right)^{k-2}\left(\begin{array}{cc}m & 1 \\ 1 & 0 \\ \end{array}\right) &=& \left(\begin{array}{cc} F_{k+1} & F_{k} \\ F_{k-1} & F_{k-2} \\ \end{array}\right)\left(\begin{array}{cc}m & 1 \\ 1 & 0 \\ \end{array}\right) \\
    &=& \left(\begin{array}{cc} \tilde{F}_{k+2}(m) & F_{k+1} \\ \tilde{F}_{k}(m) & F_{k-1} \\ \end{array}\right)
\end{eqnarray*}
from which we deduce that $[2,1^{(k-2)},m]=\frac{\tilde{F}_{k+2}(m)}{\tilde{F}_{k}(m)}.$
%\begin{eqnarray*}
%% \nonumber to remove numbering (before each equation)
%   &=&  \\
%   &=& 
%\end{eqnarray*}
%From the initial conditions and recurrence relation (\ref{fib1}) for $F_{n}(m)$ we have $F_{2}(m)/F_{1}(m)=m$ and $F_{3}(m)/F_{2}(m)=1+1/(F_{2}(m)/F_{1}(m))=[1,m]$. If we assume the statement (\ref{ratio1}) to be true, then
%\begin{eqnarray*}
%% \nonumber to remove numbering (before each equation)
% \frac{F_{n+2}(m)}{\Tilde{F}_{n+1}(m)} &=& 1+ \frac{F_{n}(m)}{\Tilde{F}_{n+1}(m)}\\
%    &=& 1+\cfrac{1}{ \frac{\Tilde{F}_{n+1}(m)}{F_{n}(m)}} \\
%   &=& [1,1^{(n-1)},m]\\
%   &=& [1^{(n)},m]
%\end{eqnarray*}
We can observe that all the partial quotients of $[1^{(n)},m]$ are bounded, and for $j\leq n$, its sequence of convergents $\{p_{j}/q_{j}\}$ is the same as the sequence of convergents for the continued fraction of $F_{n+1}/F_{n}$, which converges to the golden ratio. By comparison, for any positive integer $m$, $\tilde{F}_{n+1}(m)/\tilde{F}_{n}(m)$ also converges to the golden ration.
\end{proof}
We now show another interesting property of the continued fraction of $\alpha_{n}(N)$.
%\begin{con}
%Let $\alpha_{k}$ be as defined in (\ref{fib4}) for $k \geq1$
%\end{con}

\begin{thm}\label{thm1}
For $n\geq1$, let $$\lambda_{n}(N):=F_{n}\cdot[\overline{2N,\;1^{(n-1)}}]$$ where the $1^{(n)}$ means that $1$ has been repeated $n$ times, and $F_{n}$ is the $n$th Fibonacci number. Then $\lambda_{n}(N)$ is an algebraic integer, and for $n$ odd
$$\lambda_{n}(N) = [\tilde{F}_{n+1}(2N),\overline{2\tilde{F}_{n+1}(N)}]$$ while for $n$ even
$$\lambda_{n}(N) = [\tilde{F}_{n+1}(2N)-1,\;\overline{1,\;2\tilde{F}_{n+1}(N)-2}].$$
%with $$\frac{\beta^{'}_{k+2}}{\beta^{'}_{k}}=[2,\;1^{(k-1)},\;N]$$ a finite continued fraction.
\end{thm}

One way of approaching this question is by the use of Chatelet algorithm for integer multiples of a continued fraction described by Cusik in \cite{Cus} in order to get the continued fraction of $\lambda_{n}(N)$. This approach is tedious and very involving.  We will employ a rather straightforward approach by using the correspondence (\ref{con2}).

Let $x=[\overline{2N,1^{(n-1)}}]$ and recall
 $$x=[2N,1^{(n-1)},x]=\frac{xp_{n}+p_{n-1}}{xq_{n}+q_{n-1}}$$ where $[2N,1^{(n-1)}]=\frac{p_{n}}{q_{n}}$. By the correspondence (\ref{con2}),
\begin{eqnarray*}
% \nonumber to remove numbering (before each equation)
   \left(\begin{array}{cc}2N & 1 \\ 1 & 0 \\ \end{array}\right)\left(\begin{array}{cc}1 & 1 \\ 1 & 0 \\ \end{array}\right)^{n-1} &=& \left(\begin{array}{cc}2N & 1 \\ 1 & 0 \\ \end{array}\right)\left(\begin{array}{cc}F_{n} & F_{n-1} \\ F_{n-1} & F_{n-2} \\ \end{array}\right) \\
   &=& \left(\begin{array}{cc} F_{n+1}(2N) & F_{n}(2N) \\ F_{n} & F_{n-1} \\ \end{array}\right)
\end{eqnarray*}
we get $$[2N,1^{(n-1)}]=\frac{\tilde{F}_{n+1}(2N)}{F_{n}}$$ and
$$x=\frac{x\tilde{F}_{n+1}(2N)+\tilde{F}_{n}(2N)}{x\tilde{n}_{n}+F_{n-1}}.$$
Clearly $x$ is the largest root of the quadratic $$F_{n}x^{2}-2NF_{n}x-\tilde{F}_{n}(2N)$$ given by
$$x=N+\sqrt{N^{2}+\tilde{F}_{n}(2N)/F_{n}}.$$ $x$ is an algebraic integer and so is $\lambda_{n}(N)=F_{n}x.$ Also, $F_{n}x =\alpha_{n}(N)$ as defined in (\ref{fib4}) and whose continued fraction has been described in theorem \ref{alpha}. %The last part of  the theorem follows from observing $\beta^{'}_{k}=\tilde{F}_{k+1}(2N)+F_{k-1}=2F_{k+1}(N)$ and lemma \ref{lem2}.

%%%%%%%%%%%%%%%%%%%%%%%%%%%%%%%%%%%%%%%%%%%%%%%%%%%%%%%%%%%%%%%%%%%%%%%%%%%%%%%%%%%%%%%%%%%%%%%%%%%%%%%%%%%%%%%%%%%%%%%%%%%%%%%%%%%%%%%%%%%%%%%%%%%%%%%%%%%

\section{Lucas-like numbers and quadratic irrationals}\label{sec3}

Let $m$ be a positive integer and define the Lucas-like sequence $\tilde{L}_{n}(m)$  by
\begin{equation}\label{luc1}
  \tilde{L}_{n}(m)=mL_{n-1}+L_{n-2},
\end{equation}
where $\tilde{L}_{n}(1)$ is the usual Lucas sequence. The numbers $\tilde{L}_{n}(m)$ solves the recurrence relation
\begin{equation}\label{luc2}
  \tilde{L}_{n}(m)= \tilde{L}_{n-1}(m)+ \tilde{L}_{n-2}(m)
\end{equation}
with initial conditions $\tilde{L}_{0}(m)=3-m$,  $\tilde{L}_{1}(m)=-1+2m$ and $\tilde{L}_{2}(m)=2+m$.

For a fixed positive integer $N$, consider the quadratic
\begin{equation}\label{luc3}
  x^{2}-2NL_{n}x-L_{n}\tilde{L}_{n}(2N).
\end{equation}
The roots of the quadratic are given by
\begin{eqnarray}\label{luc4}
  \nonumber\beta_{n}(N) &=& NL_{n}+\sqrt{N^{2}L_{n}^{2}+L_{n}\tilde{L}_{n}(2N)},\\
  \overline{\beta}_{n}(N) &=& NL_{n}-\sqrt{N^{2}L_{n}^{2}+L_{n}\tilde{L}_{n}(2N)}.
\end{eqnarray}
Just like in the previous section, we are going to examine the continued fractions of these quadratic irrationals. We will use the well known identity
\begin{equation}\label{luc5}
  L_{n-1}^{2}-L_{n}L_{n-2}=(-1)^{n-1}5,
\end{equation}
as well as identities that relate the Fibonacci sequence to the Lucas sequence
\begin{equation}\label{id1}
  L_{n}=F_{n+1}+F_{n-1}
\end{equation}
and
\begin{equation}\label{id2}
  L_{n}+2L_{n-1}=5F_{n}.
\end{equation}
%######################################################################################################################################################
First we show that for $n$ even, $1/(\beta_{n}(N)-\tilde{L}_{n+1}(2N))$ is reduced.

Let $C_{n}(N)=N^{2}L_{n}^{2}+L_{n}\tilde{L}_{n}(2N)$,
\begin{eqnarray*}
% \nonumber to remove numbering (before each equation)
  1/(\beta_{n}(N)-\tilde{L}_{n+1}(2N)) &=& -1/\left(NL_{n}+L_{n-1}-\sqrt{C_{n}(N)}\right) \\
  &=&  \frac{-\left(NL_{n}+L_{n-1}+\sqrt{C_{n}(N)}\right)}{L_{n-1}^{2}-L_{n}L_{n-2}} \\
   &=& \frac{1}{5}\left(NL_{n}+L_{n-1}+\sqrt{C_{n}(N)}\right) >1.
\end{eqnarray*}
Where we have used (\ref{luc5}). On the other hand,
$$\tilde{L}_{n+1}(2N)-\overline{\beta_{n}(N)}=NL_{n}+L_{n-1}+\sqrt{C_{n}(N)} >1.$$
%We have used the identities (\ref{fib2a}) and $F_{n-1}^{2}-F_{n}F_{n-2}=(-1)^{n}$ which is easy to prove by induction on $n$ and by the recurrence relation for Fibonacci numbers.

Let $N=5k+3$ for an integer $k\geq0$, and $y_{n}(N) = 1/(\beta_{n}(N)-\tilde{L}_{n+1}(2N))$. Using (\ref{luc1}) and the Euclidean algorithm we get
%\begin{eqnarray*}
%% \nonumber to remove numbering (before each equation)
%  y_{n} &=& 2kL_{n}+2F_{n+1}+\frac{1}{5}\left(\sqrt{C_{n}(N)}-(5k+3)L_{n}-L_{n-1})\right) \\
%   &=& 2kL_{n}+2F_{n+1}+\frac{-(L_{n-1}^{2}-L_{n}L_{n-2})}{5((5k+3)L_{n}+L_{n-1}+\sqrt{C_{n}(N)})} \\
%   &=& 2kL_{n}+2F_{n+1}+\frac{1}{(5k+3)L_{n}+L_{n-1}+\sqrt{C_{n}(N)}}.
%\end{eqnarray*}
%We also have that
%\begin{eqnarray*}
%% \nonumber to remove numbering (before each equation)
%  (5k+3)L_{n}+L_{n-1}+\sqrt{C_{n}(N)} &=& 2\tilde{L}_{n+1}(2(5k+3))+\left(\sqrt{C_{n}(N)}-(5k+3)L_{n}-L_{n-1})\right) \\
%    &=&2\tilde{L}_{n+1}(2(5k+3))+\frac{-(L_{n-1}^{2}-L_{n}L_{n-2})}{(5k+3)L_{n}+L_{n-1}+\sqrt{C_{n}(N)}} \\
%    &=& 2\tilde{L}_{n+1}(2(5k+3))+\frac{1}{\frac{1}{5}\left((5k+3)L_{n}+L_{n-1}+\sqrt{C_{n}(N)}\right)} \\
%    &=& 2\tilde{L}_{n+1}(2(5k+3))+\frac{1}{y_{n}}.
%\end{eqnarray*}
\begin{eqnarray*}
% \nonumber to remove numbering (before each equation)
  y_{n}(N) &=& 2kL_{n}+2F_{n+1}+\frac{1}{5}\left(\sqrt{C_{n}(N)}-NL_{n}-L_{n-1})\right) \\
   &=& 2kL_{n}+2F_{n+1}+\frac{-(L_{n-1}^{2}-L_{n}L_{n-2})}{5\left(NL_{n}+L_{n-1}+\sqrt{C_{n}(N)}\right)} \\
   &=& 2kL_{n}+2F_{n+1}+\frac{1}{NL_{n}+L_{n-1}+\sqrt{C_{n}(N)}}.
\end{eqnarray*}
We also have that
\begin{eqnarray*}
% \nonumber to remove numbering (before each equation)
  NL_{n}+L_{n-1}+\sqrt{C_{n}(N)} &=& 2\tilde{L}_{n+1}(2N)+\left(\sqrt{C_{n}(N)}-NL_{n}-L_{n-1})\right) \\
    &=& 2\tilde{L}_{n+1}(2N)+\frac{-(L_{n-1}^{2}-L_{n}L_{n-2})}{NL_{n}+L_{n-1}+\sqrt{C_{n}(N)}} \\
 % &=& 2\tilde{L}_{n+1}(2N)+\frac{1}{\frac{1}{5}\left(NL_{n}+L_{n-1}+\sqrt{C_{n}(N)}\right)} \\
    &=& 2\tilde{L}_{n+1}(N)+\frac{1}{y_{n}(N)}.
\end{eqnarray*}
We now have that for $n$ even,
\begin{eqnarray*}
% \nonumber to remove numbering (before each equation)
  \beta_{n}(N) &=& 5y_{n}(N)-L_{n-1}\\
   &=&\tilde{L}_{n+1}(2N)+\cfrac{1}{2kL_{n}+2F_{n+1}+\cfrac{1}{2\tilde{L}_{n+1}(N)+\cfrac{1}{y_{n}(N)}}} \\
   &=& [\tilde{L}_{n+1}(2N),\;\overline{2kL_{n}+2F_{n+1},\;2\tilde{L}_{n+1}(N) }\;].
\end{eqnarray*}

For $n$ odd, the calculations for the continued fraction expansion of $\beta_{n}$ are a bit delicate due to identity (\ref{luc5}).
\begin{eqnarray*}
% \nonumber to remove numbering (before each equation)
  1/(\beta_{n}(N)-\tilde{L}_{n+1}(2N)+1) &=& \frac{1}{1-\left(NL_{n}+L_{n-1}-\sqrt{C_{n}(N)}\right)} \\
  &=&  \frac{NL_{n}+L_{n-1}+\sqrt{C_{n}(N)}}{NL_{n}+L_{n-1}+\sqrt{C_{n}(N)}-(L_{n-1}^{2}-L_{n}L_{n-2})} \\
   &=& \frac{NL_{n}+L_{n-1}+\sqrt{C_{n}(N)}}{NL_{n}+L_{n-1}+\sqrt{C_{n}(N)}-5} >1
\end{eqnarray*}
for $n>1$. We also have that
$$\tilde{L}_{n+1}(2N)-1-\overline{\beta_{n}}=NL_{n}+L_{n-1}+\sqrt{C_{n}(N)} -1>1.$$
It can easily be verified that for all $n\geq 1$, $1<y_{n}(N)<2$. Let $N=5k+3$ where $k\geq0$ is an integer, and
$$y_{n}(N) = 1/\left(1-\left(NL_{n}+L_{n-1}-\sqrt{C_{n}(N)}\right)\right).$$ Then,
\begin{eqnarray*}
% \nonumber to remove numbering (before each equation)
  y_{n}(N) &=& 1+\frac{1}{\frac{1}{5}\left(NL_{n}+L_{n-1}+\sqrt{C_{n}(N)}-5\right)} \\
    &=& 1+\frac{1}{2kL_{n}+2F_{n+1}-2+\frac{1}{5}\left(5-NL_{n}-L_{n-1}+\sqrt{C_{n}(N)}\right)}.% \\
%    &=& 1+\frac{1}{NL_{n}+L_{n-1}+\sqrt{C_{n}(N)}-1}   \\
%    &=&  1+\frac{1}{2\tilde{L}_{n+1}(N)-2+(1-(NL_{n}+L_{n-1}-\sqrt{L_{n}}))}  \\
%    &=&  1+\frac{1}{2\tilde{L}_{n+1}(N)-2+1/y_{n}}.
\end{eqnarray*}
It is easy to show (by induction on $n$) that $$0<\frac{1}{5}\left(5-NL_{n}-L_{n-1}+\sqrt{C_{n}(N)}\right)<1.$$
Now
\begin{eqnarray*}
% \nonumber to remove numbering (before each equation)
  \frac{1}{\frac{1}{5}\left(5-NL_{n}-L_{n-1}+\sqrt{C_{n}(N)}\right)} &=& 1+\frac{1}{NL_{n}+L_{n-1}+\sqrt{C_{n}(N)}-1} \\
    &=&  1+\cfrac{1}{2\tilde{L}_{n+1}(N)-2+\cfrac{1}{y_{n}(N)}}
\end{eqnarray*}
This implies that for $n$ odd,
\begin{eqnarray*}
% \nonumber to remove numbering (before each equation)
  \beta_{n}(N) &=& \tilde{L}_{n+1}(2N)-1+\frac{1}{y_{n}(N)} \\
   &=& [\tilde{L}_{n+1}(2N)-1,\;\overline{1,\;2kL_{n}+2F_{n+1}-2,\;1\;2\tilde{L}_{n+1}(N)-2}].
\end{eqnarray*}
We can also make the simplification
\begin{eqnarray*}
% \nonumber to remove numbering (before each equation)
  2kL_{n}+2F_{n+1} &=& 2kL_{n}+\frac{2}{5}(3L_{n}+L_{n-1}) \\
    &=& \frac{2}{5}((5k+3)L_{n}+L_{n-1})  \\
    &=&  \frac{2}{5}L_{n+1}(N)
\end{eqnarray*}
 in the two statements above.  We have proved

\begin{thm}\label{beta1}
Let $N=5k+3$ for $k\geq0$ and  $\beta_{n}(k)$ be the largest root of the quadratic (\ref{luc3}), then the continued fraction of $\beta_{n}(k)$ for $n$ odd is given by
\begin{equation}\label{luc6}
  \beta_{n}(N)=[\tilde{L}_{n+1}(2N)-1,\overline{1,\frac{2}{5}L_{n+1}(N)-2,1,2\tilde{L}_{n+1}(N)-2}]
\end{equation}
while for $n$ even it is given by
\begin{equation}\label{luc7}
  \beta_{n}(N)=[\tilde{L}_{n+1}(2N),\overline{\frac{2}{5}L_{n+1}(N),2\tilde{L}_{n+1}(N)}].
\end{equation}
\end{thm}

%######################################################################################################################################################
There are some purely periodic continued fractions with periods of arbitrary length which when multiplied by $L_{n}$, give periodic continued fractions with period of length $2$ or $4$ depending on the parity of $n$. We give some of these continued fractions below.

\begin{thm}\label{beta2}
Let $n\geq2$, $k\geq0$ and $$\mu_{n}(k):=L_{n}\cdot[\overline{2(5k+3),1^{(n-2)},2,1,2k,1,2,1^{(n-2)}}]$$ where the $1^{(n)}$ means that $1$ has been repeated $n$ times, and $L_{n}$ is the $n$th Lucas number. Then $\mu_{n}(k)$ is an algebraic integer, and for $n$ odd
\begin{equation}\label{luc6}
  \mu_{n}(k)=[\tilde{L}_{n+1}(10k+6)-1,\overline{1,\frac{2}{5}L_{n+1}(5k+3)-2,1,2\tilde{L}_{n+1}(5k+3)-2}]
\end{equation}
while for $n$ even it is given by
\begin{equation}\label{luc7}
  \mu_{n}(k)=[\tilde{L}_{n+1}(10k+6),\overline{\frac{2}{5}L_{n+1}(5k+3),2\tilde{L}_{n+1}(5k+3)}].
\end{equation}
\end{thm}

\begin{proof}
Fix $N=5k+3$ and let $$y=[\overline{2N,1^{(n-2)},2,1,2k,1,2,1^{(n-2)}}].$$
Then $$y= [2N,1^{(n-2)},2,1,2k,1,2,1^{(n-2)},y]=\frac{yp_{n}+p_{n-1}}{yq_{n}+q_{n-1}}$$
where $[2N,1^{(n-2)},2,1,2k,1,2,1^{(n-2)},y]=\frac{p_{n}}{q_{n}}$.
 By the correspondence (\ref{con2}),
 $$\resizebox{0.98\hsize}{!}{$\left(\begin{array}{cc}2N & 1 \\ 1 & 0 \\ \end{array}\right)\left(\begin{array}{cc}1 & 1 \\ 1 & 0 \\ \end{array}\right)^{n-2}\left(\begin{array}{cc}2 & 1 \\ 1 & 0 \\ \end{array}\right)\left(\begin{array}{cc}1 & 1 \\ 1 & 0 \\ \end{array}\right)\left(\begin{array}{cc}2k & 1 \\ 1 & 0 \\ \end{array}\right)\left(\begin{array}{cc}1 & 1 \\ 1 & 0 \\ \end{array}\right)\left(\begin{array}{cc}2 & 1 \\ 1 & 0 \\ \end{array}\right)\left(\begin{array}{cc}1 & 1 \\ 1 & 0 \\ \end{array}\right)^{(n-2)}$}$$
$$
$$
\begin{eqnarray*}
% \nonumber to remove numbering (before each equation)
    &=& \left(\begin{array}{cc}2N & 1 \\ 1 & 0 \\ \end{array}\right)\left(\begin{array}{cc}F_{n-1} & F_{n-2} \\ F_{n-2} & F_{n-3} \\ \end{array}\right)\left(\begin{array}{cc}6(3k+2) & 6k+5 \\ 6k+5 & 2k+2 \\ \end{array}\right)\left(\begin{array}{cc}F_{n-1} & F_{n-2} \\ F_{n-2} & F_{n-3} \end{array}\right) \\
   &=& \left(\begin{array}{cc} p_{n} &p_{n-1} \\ q_{n} & q_{n-1} \\ \end{array}\right)
\end{eqnarray*}
with $p_{n},\;p_{n-1},\; q_{n}$ and $q_{n-1}$ given by
\begin{eqnarray*}
% \nonumber to remove numbering (before each equation)
  p_{n}  &=& \frac{4}{5}N^{2}L_{n}^{2}+\frac{6}{5}NL_{n}L_{n-1}+\frac{1}{5}\left(L_{n}^{2}-L_{n}L_{n-1}+L_{n-1}^{2}\right), \\
  p_{n-1}  &=& \frac{4}{5}N^{2}L_{n}L_{n-1}+\frac{2}{5}N\left(L_{n}^{2}-L_{n}L_{n-1}\right)+\frac{2}{5}L_{n}L_{n-1}+(4k+2)L_{n-1}^{2}, \\
          &=& \frac{2}{5}(2N-1)L_{n-1}\tilde{L}_{n+1}(N)+\frac{2}{5}L_{n}\tilde{L}_{n+1}(N), \\
          &=& \frac{2}{5}\tilde{L}_{n+1}(N)\left(2NL_{n-1}-L_{n-1}+L_{n}\right), \\
          &=& \frac{2}{5}\tilde{L}_{n+1}(N)\tilde{L}_{n}(2N), \\
  q_{n}  &=& \frac{2}{5}NL_{n}^{2}+\frac{2}{5}L_{n}L_{n-1}, \\
         &=& \frac{2}{5}L_{n}\tilde{L}_{n+1}(N) \\
  q_{n-1}  &=& (2k+1)L_{n}L_{n-1}+\frac{1}{5}\left(L_{n}^{2}+L_{n-1}^{2}\right).
\end{eqnarray*}
In evaluating the matrix multiplication, we used the identities (\ref{luc1}), (\ref{id1}) and (\ref{id2}). Clearly, $y$ is the largest root of the quadratic $$y^{2}q_{n}+(q_{n-1}-p_{n})y-p_{n-1}.$$
Here, $$q_{n-1}-p_{n}=-\frac{4}{5}NL_{n}\tilde{L}_{n+1}(N)$$ and so the quadratic can also be written as
$$\frac{2}{5}\tilde{L}_{n+1}(N)\left(y^{2}L_{n}-2NL_{n}y-\tilde{L}_{n}(2N)\right)$$

The largest root of the quadratic is given by $$y=\frac{1}{L_{n}}\left(NL_{n}+\sqrt{N^{2}L_{n}^{2}+L_{n}\tilde{L}_{n}(2N)}\right).$$ Clearly, $y$ is an algebraic integer, and so is $\mu_{n}(k)=L_{n}y$. The continued fraction of $\mu_{n}(k)=L_{n}y$ follows from theorem \ref{beta1}.
\end{proof}

%5555555555555555555555555555555555555555555555555555555555555555555555555555555555555555555555555555555555555555555555555555555555555555555555555555555555

\section{Generalizations}\label{sec4}
%For $k\geq 1$ consider the Fibonacci polynomials $F_{k}(x)$ generated by  the rational function
%\begin{equation}\label{fibgen}
%  \frac{1}{1-xt-t^{2}}=\sum_{k=1}^{\infty}F_{k}(x)t^{k}.
%\end{equation}
%Our aim is to
We now examine the continued fraction of $G_{k}(N,x)$ defined by
\begin{equation}\label{conf1}
  G_{k}(N,x):=F_{k}(x)\cdot[\overline{2N,x^{(k-1)}}]
\end{equation}
where $x\geq1$, $x^{(k)}=x,x,\dots,x$ repeated $k$ times and $n$ is a non-zero positive integer.

Let $\tilde{G}_{k}(N,x)=[\overline{2N,x^{(k-1)}}]$ so that $$\tilde{G}_{k}(N,x)=[2N,x^{(k-1)},\tilde{G}_{k}(N,x)]=\frac{\tilde{G}_{k}(N,x)p_{k}(x)+p_{k-1}(x)}{\tilde{G}_{k}(N,x)q_{k}(x)+q_{k-1}(x)}$$ where $[2N,x^{(k-1)}]=\frac{p_{k}(x)}{q_{k}(x)}$.
It can easily be shown by induction on $k$ that %for all $k\geq 1$,
 $$\left(\begin{array}{cc}x & 1 \\ 1 & 0 \\ \end{array}\right)^{k}=\left(\begin{array}{cc} F_{k+1}(x) & F_{k}(x) \\ F_{k}(x) & F_{k-1}(x) \\ \end{array}\right),$$
from which we get
\begin{eqnarray*}
% \nonumber to remove numbering (before each equation)
   \left(\begin{array}{cc}2N & 1 \\ 1 & 0 \\ \end{array}\right)\left(\begin{array}{cc}x & 1 \\ 1 & 0 \\ \end{array}\right)^{k-1} &=& \left(\begin{array}{cc}2N & 1 \\ 1 & 0 \\ \end{array}\right)\left(\begin{array}{cc}F_{k}(x) & F_{k-1}(x) \\ F_{k-1}(x) & F_{k-2}(x) \\ \end{array}\right) \\
   &=& \left(\begin{array}{cc}2N F_{k}(x)+F_{k-1}(x) & 2NF_{k-1}(x)+F_{k-2}(x) \\ F_{k}(x) & F_{k-1}(x) \\ \end{array}\right).
\end{eqnarray*}
This implies that $$[2N,x^{(k-1)}]=\frac{2N F_{k}(x)+F_{k-1}(x)}{F_{k}(x)}$$ and
$$\tilde{G}_{k}(N,x)=\frac{\tilde{G}_{k}(N,x)\left( 2N F_{k}(x)+F_{k-1}(x) \right)+2NF_{k-1}(x)+F_{k-2}(x)}{\tilde{G}_{k}(N,x)F_{k}(x)+F_{k-1}(x)}.$$
 $\tilde{G}_{k}(N,x)$ is the largest root of the quadratic $$F_{k}(x)z^{2}-2NF_{k}(x)z-(2NF_{k-1}(x)+F_{k-2}(x))$$ given by $$\tilde{G}_{k}(N,x)=N+\sqrt{N^{2}+(2NF_{k-1}(x)+F_{k-2}(x))/F_{k}(x)}. $$ We can now write
  \begin{equation}\label{idf1}
   G_{k}(N,x)=NF_{k}(x)+\sqrt{N^{2}F_{k}^{2}(x)+F_{k}(x)(2NF_{k-1}(x)+F_{k-2}(x))}.
 \end{equation}

 As we will see below, the continued fraction of $G_{k}(N,x)$ depends on the parity of $k$ because of the identity
 \begin{equation}\label{idf2}
   F_{k-1}^{2}(x)-F_{k}(x)F_{k-2}(x)=(-1)^{k}.
 \end{equation}
To simplify the notation, let $$\beta_{k}(N,x)=N^{2}F_{k}^{2}(x)+F_{k}(x)(2NF_{k-1}(x)+F_{k-2}(x))$$ so that $G_{k}(N,x)=NF_{k}(x)+\sqrt{\beta_{k}(N,x)}.$ We first examine the continued fraction of $G_{k}(N,x)$ for $k$ odd.
For all $k\geq1$ and $x\geq1$, we have by  (\ref{idf2}) that $\beta_{k}(N,x)=(NF_{k}(x)+F_{k-1}(x))^{2}+1$ so that $$G_{k}(N,x)-(2NF_{k}(x)+F_{k-1}(x))=\sqrt{\beta_{k}(N,x)}-(NF_{k}(x)+F_{k-1}(x))>0.$$
%as a consequence of (\ref{idf2}) and
\begin{eqnarray*}
% \nonumber to remove numbering (before each equation)
 1/(G_{k}(N,x)-(2NF_{k}(x)+F_{k-1}(x))) &=& \frac{NF_{k}(x)+F_{k-1}(x)+\sqrt{\beta_{k}(N,x)}}{-(F_{k-1}^{2}(x)-F_{k}(x)F_{k-2}(x))}\\
  &=& NF_{k}(x)+F_{k-1}(x)+\sqrt{\beta_{k}(N,x)}>0.
\end{eqnarray*}
%where we have used (\ref{idf2}).
Call this $G_{k}^{(1)}(N,x).$ $$G_{k}^{(1)}(N,x)-(2NF_{k}(x)+2F_{k-1}(x))=\sqrt{\beta_{k}(N,x)}-(NF_{k}(x)+F_{k-1}(x))$$ and from the above calculations,
$$1/(G_{k}^{(1)}(N,x)-(2NF_{k}(x)+2F_{k-1}(x)))=G_{k}^{(1)}(N,x).$$
Hence the continued fraction of $G_{k}^{(1)}(N,x)$ is purely periodic with period of length 1 given by $\{2NF_{k}(x)+2F_{k-1}(x)\}$. For all $k>1$ odd, the continued fraction of $G_{k}(N,x)$ is given by,
\begin{equation}\label{conf2}
  G_{k}(N,x)=[2NF_{k}(x)+F_{k-1}(x),\;\overline{2NF_{k}(x)+2F_{k-1}(x)}].
\end{equation}

\noindent Now for the case of $k>1$ even,  $\beta_{k}(N,x)=(NF_{k}(x)+F_{k-1}(x))^{2}-1$  by  (\ref{idf2}), so that
$$G_{k}(N,x)-(2NF_{k}(x)+F_{k-1}(x)-1)=\sqrt{\beta_{k}(N,x)}-(NF_{k}(x)+F_{k-1}(x))+1>0.$$
Call this $G_{k}^{(2)}(N,x).$
\begin{eqnarray*}
% \nonumber to remove numbering (before each equation)
 1/G_{k}^{(2)}(N,x) &=& 1+\frac{NF_{k}(x)+F_{k-1}(x)-\sqrt{\beta_{k}(N,x)}}{G_{k}^{(2)}(N,x)} \\
    &=& 1+\frac{F_{k-1}^{2}(x)-F_{k}(x)F_{k-2}(x)}{NF_{k}(x)+F_{k-1}(x)+\sqrt{\beta_{k}(N,x)}-(F_{k-1}^{2}(x)-F_{k}(x)F_{k-2}(x))}   \\
    &=& 1+\frac{1}{NF_{k}(x)+F_{k-1}(x)+\sqrt{\beta_{k}N,(x)}-1}   \\
    &=&  1+\frac{1}{2NF_{k}(x)+2F_{k-1}(x)-2+G_{k}^{(2)}(N,x)}
    %&=&  1+\frac{1}{F_{n+1}(2N)+F_{n-1}-2+1/y_{n}}.
\end{eqnarray*}
We can now see that $1/G_{k}^{(2)}(N,x)$  is purely periodic with a period of length 2 given by  $\{1,\;2NF_{k}(x)+2F_{k-1}(x)-2\}$. For all $k>1$ even, the continued fraction of $G_{k}(N,x)$ is given by,
\begin{equation}\label{conf2}
  G_{k}(N,x)=[2NF_{k}(x)+F_{k-1}(x)-1,\;\overline{1,\;2NF_{k}(x)+2F_{k-1}(x)-2}].
\end{equation}

We have proved that
\begin{thm}
For all $k\geq1$ and $x\geq1$, the product of the $k$th Fibonacci polynomial $F_{k}(x)$ with the periodic continued fraction $[\overline{2N,\;x^{(k-1)}}]$ gives the periodic continued fraction
$$  [2N F_{k}(x)+F_{k-1}(x),\overline{2N F_{k}(x)+2F_{k-1}(x)}] $$ for $k$ odd and
$$ [2N F_{k}(x)+F_{k-1}(x)-1,\overline{1,2N F_{k}(x)+2F_{k-1}(x)-2}] $$ for $k$ even.
\end{thm}

If we set $x=1$ we get the result of theorem \ref{thm1}.

%5555555555555555555555555555555555555555555555555555555555555555555555555555555555555555555555555555555555555555555555555555555555555555555555555555555555

\section{Polynomials from convergents of $[\overline{N,x^{(k)}}]$}\label{sec5}
In this section, we show how polynomials arising from the convergents of $[\overline{N,x^{(k)}}]$ are related to Chebyshev and Fibonacci polynomials. We describe the polynomials for $k=1,2$ and $3$ and show how the roots of these polynomials are distributed.

First, for a fixed $k=1$, let $p_{n}(N,x)/q_{n}(N,x)$ be the convergents of $[\overline{N,x}]$. When $n\equiv(0\mod 2)$ and $n\equiv(1\mod 2)$, $q_{2n}(N,x)$ and $q_{2n+1}(N,x)$ are respectively generated by the rational functions
\begin{equation}\label{gen1}
  \frac{1-t}{1-(Nx+2)t+t^{2}}\;\;\;\text{ and }\;\;\;\frac{x}{1-(Nx+2)t+t^{2}}.
\end{equation}
Comparing these generating functions to the generating function of Chebyshev polynomials of the second kind below
\begin{equation}\label{gen1.1}
  \frac{1}{1-2xt+t^{2}}=\sum_{n=0}^{\infty}U_{n}(x)t^{n},
\end{equation}
we have $$q_{2n}(N,x)=U_{n}(N/2x+1)-U_{n-1}(N/2x+1)\;\;\;\;\text{ and }\;\;\;\;q_{2n+1}(N,x)=xU_{n}(N/2x+1).$$
To determine the roots of $q_{k}(x)$, first note that $U_{n}(x)$ has all its roots in the interval $(-1,1)$ given by $x_{k}=\cos\left(\frac{k}{n+1}\pi\right)$, see for example \cite[section 2.2]{Mason}. It follows that $q_{2n+1}(N,x)$ has real roots in the interval $(-4/N,0]$ and are given by $x_{k}=\frac{2}{N}\left(\cos\left(\frac{k}{n+1}\pi\right)-1\right).$ It can also be shown that $q_{2n}(N,x)$ has $n-1$ real roots in the interval $(-4/N,0]$ and one real root outside this interval.

Now for the case when $k=2$, let $p_{n}(N,x)/q_{n}(N,x)$ be the convergents of $[\overline{N,x,x}]$. When $n\equiv(0\mod 3)$ and $n\equiv(1\mod 3)$, $q_{3n}(N,x)$ and $q_{3n+1}(N,x)$ are respectively generated by the rational functions
\begin{equation}\label{gen1}
  \frac{x+t}{1-(Nx^{2}+2x+N)t-t^{2}}\;\;\;\text{ and }\;\;\;\frac{x}{1-((Nx^{2}+2x+N))t-t^{2}}
\end{equation}
and their zeros seem to lie close to the hyperbola $y^{2}-x^{2}=1-1/N^{2}$. %$q_{3n+2}(N,x)$ are generated by $$\frac{x^{2}+1}{1-(Nx^{2}+2x+N)t-t^{2}}$$

%Comparing these generating functions to the generating function of Chebyshev polynomials of the second kind below
%\begin{equation}\label{gen1.1}
%  \frac{1}{1-2xt+t^{2}}=\sum_{n=0}^{\infty}U_{n}(x)t^{n},
%\end{equation}
%
%When $n\equiv(0\mod3)$, $q_{n}(x)$ has a factor of $x^{2}+1$ for all $n\geq3$. By eliminating this  factor and making a change of variable $x\mapsto x-1/3$ we get (after clearing denominators) polynomials $Q_{n}(x)$ that are generated by
%\begin{equation}\label{gen2}
%  \frac{1}{1-(9x^{2}+8)t-9t^{2}}=\sum_{n=0}^{\infty}Q_{n}(x)t^{n}.
%\end{equation}
%\begin{thm}
%Let $Q_{n}(x)$ be as defined in equation (\ref{gen2}). Then for all $n\geq 1$, all the zeros $Q_{n}(x)$ lie on the hyperbola $$y^{2}-x^{2}=\frac{8}{9}.$$
%\end{thm}
%This is just a special case of theorem \ref{hyp1} below.
%

Let $p_{n}(N,x)/q_{n}(N,x)$ be the convergents of $[\overline{N,x,x}]$. When $n\equiv(2\mod3)$, $q_{n}(N,x)$  are generated by the rational function
$$ \frac{1+x^{2}}{1-(Nx^{2}+2x+N)t-t^{2}}$$
and have a factor of $x^{2}+1$. By eliminating this  factor and making a change of variable $x\mapsto x-1/N$ we get (after clearing  denominators) polynomials $Q_{n}(N,x)$ that are generated by
\begin{equation}\label{gen3}
  \frac{1}{1-(N^{2}x^{2}+N^{2}-1)t-N^{2}t^{2}}=\sum_{n=0}^{\infty}Q_{n}(N,x)t^{n}.
\end{equation}
This shifting of the polynomial reduces the number of terms as can be seen in the generating function above. Comparing this generating function to that of the Fibonacci polynomials (\ref{fibgen}), we have
$$Q_{k}(N,x)=N^{k}F_{k}\left(Nx^{2}+N-1/N\right).$$

\begin{thm}\label{hyp1}
For any non-zero $k\in\R$, let $Q_{n}(N,x)$ be as defined in equation (\ref{gen3}). Then for all $n\geq 1$, all the zeros $Q_{n}(N,x)$ lie on the hyperbola $$H1:\;\;\;\;y^{2}-x^{2}=\frac{N^{2}-1}{N^{2}}.$$
\end{thm}
\begin{proof}
Bicknell  and Hoggatt proved in \cite{Bic} that if $F_{n}(x)=0$ then $x=2i\cos(j\pi/n)$ for $j=1,2,\dots,n-1$. Let $$z=\frac{\sqrt{N^{2}-1}}{N}(\sinh\phi+i\cosh\phi)$$
where $0<N$. Then
\begin{eqnarray*}
% \nonumber to remove numbering (before each equation)
  Nz^{2}+N-1/N &=& (N-1/N)(2i\sinh\phi\cosh\phi-1)+N-1/N  \\
    &=&  (N-1/N)i\sinh2\phi.
\end{eqnarray*}
Now $Q_{n}(N,z)=0$ implies that $$\sinh2\phi=\frac{2N}{N^{2}-1}\cos\theta_{j}$$ where $\theta_{j}=j\pi/(n+1)$ for $j=1,2,\dots,n-1$. Using the identity
\begin{equation}\label{id33}
  \sinh^{-1}x = \log(x+\sqrt{x^{2}+1}) \;\;\;\;\;-\infty<x<\infty
\end{equation}
in which we consider the principal branch of the $\log$, we get
\begin{equation}\label{exp21}
  \phi_{j}=\frac{1}{2}\log\left|2N\cos\theta_{j}+\sqrt{4N^{2}\cos^{2}\theta_{j}+(N^{2}-1)^{2}}\right|-\frac{1}{2}\log|N^{2}-1|.
\end{equation}
The $2n$ zeros of $Q_{n}(N,z)$ are now given by
\begin{equation}\label{exp22}
  z_{j}=\pm\sqrt{(1-1/N^{2})}(\sinh\phi_{j}+i\cosh\phi_{j})
\end{equation}
for $j=1,2,\dots,n-1$ where $\phi_{j}$ is given by (\ref{exp21}). It is straightforward to check that all these points lie on $H1$.
\end{proof}

Numerical calculations suggest that the polynomials generated by the rational functions (\ref{gen3}) have their roots close to $H1$.
%\begin{rem}
%This theorem also follows from a theorem of Khang Tran but we have the advantage of a formula for the roots of the polynomials. A theorem of Khang Tran can also be used to show that the roots of $q_{3n}(N,x)$ and $q_{3n+1}(N,x)$ lie close to $H1$.
%\end{rem}

These polynomials arising from the convergents of $[\overline{N,x^{(k)}}]$ seem to get more complicated as $k$ gets larger. For an example, let $k=3$, and $p_{n}(N,x)/q_{n}(N,x)$ be the convergents of $[\overline{N,x,x,x}]$ then $q_{4n+3}(N,x)$ are generated by
$$\frac{x(x^{2}+2)}{1-(Nx^{3}+2x^{2}+2Nx+2)t+t^{2}}.$$

Now for $k=4$, $q_{5n+4}(N,x)$ are generated by $$\frac{x^{4}+3^{2}+1}{1-(Nx^{4}+2x^{3}+3Nx^{2}+4x+N)t-t^{2}}.$$

%For $k$ odd and $n\equiv -1\mod k$, $q_{n}(N,K)$ seem to have a factor $Q_{n}(N,x)$ which is a `specialized' Chebyshev polynomial. On the other hand, for $k$ even and $n\equiv -1\mod k$, $q_{n}(N,K)$ seem to have a factor $Q_{n}(N,x)$ which is a `specialized' Fibonacci polynomial.

In general, for a fixed $k$, let  $q_{n}(N,x)$ be the denominator of the convergents of $[\overline{N,x^{(k)}}]$ when $n\equiv -1\mod (k+1)$. To simplify the notation, call them $Q_{m}(x)$. Then $Q_{m}(x)$ are generated by
\begin{equation}\label{genx1}
  \frac{F_{k+1}(x)}{1-g_{k}(N,x)t-(-1)^{k}t^{2}}=\sum_{m=0}^{\infty}Q_{m}(N,x)t^{m},
\end{equation}
where $g_{k}(N,x)$ are generated by the rational function
\begin{equation}\label{genx2}
  \frac{N+2t}{1-xt-t^{2}}=\sum_{k=0}^{\infty}g_{k}(N,x)t^{k}.
\end{equation}
By (\ref{fibgen}), we get the explicit expression for $g_{k}(N,x)$ as $$g_{k}(N,x)=NF_{k}(x)+2F_{k-1}(x).$$

We can apply a theorem of Khang Tran, see \cite[Theorem 1]{Tran}, to determine the curve on which the roots of $Q_{m}(N,x)$ lie. As $k$ increases however, so does the difficulty in describing this curve. As an example, for $k=4$ the roots of $Q_{m}(3,x)/F_{5}(x)$ lie on the curve given by
$$3 x^4-18 x^2 y^2+3 y^4+2 x^3-6 x y^2+9 x^2-9 y^2+4 x+3=0.$$

The relationship to Fibonacci and Chebyshev polynomials of the second kind follows from the generating function (\ref{genx1}). For a fixed $k$ that is even, $$Q_{n}(N,x)=F_{k+1}(x)F_{n}(g_{k}(N,x)),$$ while for a fixed $k$ that is odd, $$Q_{n}(N,x)=F_{k+1}(x)U_{n}(g_{k}(N,x)).$$
%This connection to Chebyshev polynomials is most puzzling.

\section*{Conclusion}
In conclusion, we pose a question. Are there any other integer sequences $f_{n}$ such that $\tilde{f}_{n+1}(m):=mf_{n}+f_{n-1}$, and for which the continued fraction of $$\frac{1}{f_{n}}\left(Nf_{n}+\sqrt{N^{2}f_{n}^{2}+f_{n}\tilde{f}_{n}(2N)}\right)$$
is purely periodic with length of the period depending on $n$. And the continued fraction of $$Nf_{n}+\sqrt{N^{2}f_{n}^{2}+f_{n}\tilde{f}_{n}(2N)}$$ is periodic with a fixed period length independent of $n$?
%$\sfrac{3}{4}$

\section*{Acknowledgements}
I would like to thank Professor Kenneth Stolarsky for pointing me towards this problem and for his continued support, insightful comments and guidance in the writing of this paper.

%$$S(N,n)=Sup_{P(\alpha)=n}^{}\;P(N\alpha)$$
\end{document}